\documentclass[11pt, leqno]{amsart}
\usepackage{amsmath}
\usepackage{amsfonts}
\usepackage{amssymb}
\usepackage{graphicx}
\usepackage{color}
\usepackage[latin1]{inputenc}
\usepackage{amsthm, %refcheck
}
\usepackage[margin=3.3cm]{geometry}

\newtheorem{thm}{Theorem}[section]

\newtheorem{prop}[thm]{Proposition}
\newtheorem{rem}[thm]{Remark}
\numberwithin{equation}{section}

\newcommand{\beq}{\begin{equation}}
\newcommand{\eeq}{\end{equation}}

\def\qed{\,\unskip\kern 6pt \penalty 500
\raise -2pt\hbox{\vrule \vbox to8pt{\hrule width 6pt
\vfill\hrule}\vrule}\par}

\definecolor{darkblue}{rgb}{0.05, .05, .65}
\definecolor{darkgreen}{rgb}{0.05, .70, .05}
\definecolor{darkred}{rgb}{0.8,0,0}
\def\qed{\unskip\kern 6pt \penalty 500
\raise -2pt\hbox{\vrule \vbox to8pt{\hrule width 6pt
\vfill\hrule}\vrule}\par}
\begin{document}

\title[Smoothing effects for the PME on Cartan-Hadamard manifolds]{\textbf{ \large Smoothing effects \\ for the porous medium equation \\ on Cartan-Hadamard manifolds}}

\author{Gabriele Grillo, Matteo Muratori}

\address{Gabriele Grillo: Dipartimento di Matematica, Politecnico di Milano, Piaz\-za Leonardo da Vinci 32, 20133 Milano, Italy}
\email{gabriele.grillo@polimi.it}

\address{Matteo Muratori: Dipartimento di Matematica ``F. Enriques'', Universit\`a degli Studi di Milano, via Cesare Saldini 50, 20133 Milano, Italy}
\email{matteo.muratori@unimi.it}

% \date{} %%  this cancels date in article format

\begin{abstract}
We prove three sharp bounds for solutions to the porous medium equation posed on Riemannian manifolds, or for weighted versions of such equation. Firstly we prove a smoothing effect for solutions which is valid on any Cartan-Hadamard manifold whose sectional curvatures are bounded above by a strictly negative constant. This bound includes as a special case the sharp smoothing effect recently proved by V\'azquez on the hyperbolic space in \cite{V}, which is similar to the absolute bound valid in the case of bounded Euclidean domains but has a logarithmic correction. Secondly we prove a bound which interpolates between such smoothing effect and the Euclidean one, supposing a suitable quantitative Sobolev inequality holds, showing that it is sharp by means of explicit examples. Finally, assuming a stronger functional inequality of sub-Poincar\'e type, we prove that the above mentioned (sharp) absolute bound holds, and provide examples of weighted porous media equations on manifolds of infinite volume in which it holds, in contrast with the non-weighted Euclidean situation. It is also shown that sub-Poincar\'e inequalities cannot hold on Cartan-Hadamard manifolds.
\end{abstract}

\maketitle

\section{Introduction}

We shall consider the following Cauchy problem for the porous medium equation (PME in the sequel)
\begin{equation}\label{eq-non-pesata}
\begin{cases}
u_t =  \Delta (u^m) & \textrm{in } M \times \mathbb{R}^{+} \, , \\
u =u_0 \in L^1(M) & \textrm{on } M \times \{ 0 \} \, .
\end{cases}
\end{equation}
where we suppose throughout this paper that $m>1$, and the equation is posed on a Riemannian manifold $M$, $\Delta$ indicating of course its (negative) Riemannian Laplacian. We use throughout the paper, without further comment, the usual convention $u^m:=|u|^{m-1}u$. We shall in particular be interested in  the case in which $M$ is a Cartan-Hadamard manifold, namely a complete and simply connected manifold of nonpositive sectional curvature. The study of such kind of evolutions has been initiated in \cite{BGV} in the fast diffusion case, i.e.\ when $m<1$. In particular, solutions corresponding to a significant class of initial data vanish in finite time for all $m<1$, in striking contrast with the Euclidean situation, where this happens only if $m<m_c=(d-2)/d$. In a sense, the setting considered has closer similarities with the case of an Euclidean \it bounded \rm domain, where similar phenomena occur (see e.g.\ \cite{DKW}, \cite{BGV2}). Recently these topics have been investigated further in \cite{GM2}, where detailed asymptotic results are proven in the most important example of negatively curved manifold, namely the $d$-dimensional hyperbolic space ${\mathbb H}^d$.

The porous medium case $m>1$ has instead been less studied till a very recent paper by J. L. V\'azquez \cite{V}. In it, a careful analysis of the nonlinear flow is carried out on ${\mathbb H}^d$. The problem is connected in the radial case, via a change of variable, to a \it weighted \rm PME-type equation in the \it Euclidean \rm setting, and several striking results on the detailed behaviour of solutions are proved. In particular, it is shown that for radial data, say, with $\|u_0\|_1=1$, the following bound holds:
\begin{equation}\label{eq: smoothing-hyp}
\|u(t)\|_\infty\le C\left(\frac{\log t}t\right)^{\frac1{m-1}}\ \ \ \forall t\ge2,
\end{equation}
where $C$ is a suitable positive constant. V\'azquez also shows that the bound \eqref{eq: smoothing-hyp} is sharp. Moreover, the evolution of the support of solutions corresponding to compactly supported data is also studied, with sharp asymptotic results, in \cite{V}. Weaker results on this latter problem were given previously in \cite{Pu}.

The fact that the solution $u(t)$ is bounded at all $t>0$ is not difficult to prove as a consequence of the Sobolev or Gagliardo-Nirenberg inequalities (see e.g.\  \cite{BG}, \cite{GM}, \cite{GM3}), which are well known to hold in ${\mathbb H}^d$. However, the quantitative form of the bound, for large $t$, is significantly different from its Euclidean counterpart and from what could be deduced from the above mentioned functional inequalities. In fact, the (best possible) Euclidean upper bound involves the dimension-dependent quantity $t^{-d/[d(m-1)+2]}$. Notice that \[d/[d(m-1)+2]<1/(m-1),\] hence the decay rate predicted by \eqref{eq: smoothing-hyp} is \it faster \rm than its Euclidean counterpart, in analogy with a similar well-known property for the linear heat equation on ${\mathbb H}^d$. Note also that the bound in \eqref{eq: smoothing-hyp} differs only by a logarithmic correction from the one valid for solutions $v(t)$ to the PME on bounded Euclidean domains with homogeneous Dirichlet boundary conditions, namely:
\[
\|v(t)\|_\infty\le \frac C{t^{\frac1{m-1}}}.
\]
Again, the smoothing effect for the PME on ${\mathbb H}^d$ is in some sense closer to the one valid in bounded Euclidean domains than it is to the one valid on the whole Euclidean space.

\smallskip
Our goal here is to investigate in more detail bounds similar to \eqref{eq: smoothing-hyp}. In fact, we shall show in Section \ref{SE} that the smoothing effect \eqref{eq: smoothing-hyp} is true in any Cartan-Hadamard manifold whose sectional curvature is bounded above by a suitable constant $-k<0$. In fact, the proof will rely only on functional inequalities, namely on the combination of the Sobolev inequality and of the Poincar\'e-type, or spectral gap, inequality
\[
\|\nabla f\|^2_2\ge \Lambda \|f\|_2^2\ \ \ \forall f \in C_c^\infty(M) \, .
\]
Both inequalities are in fact known to hold under the stated geometric assumption. We thus enlarge considerably the class of manifolds on which \eqref{eq: smoothing-hyp} is valid, and connect such bound with clear geometric hypotheses. Of course the result is sharp, since this fact is true already on ${\mathbb H}^d$. The result is contained in Theorem \ref{teo: log-1}.

\smallskip
In Section \ref{deg} we deal with an apparently technical modification of the above result. We show that, if a family of Sobolev inequalities involving $L^q$ norms with $q\in(2,2^*]$, $2^*$ being the usual Sobolev exponent, holds with a proportionality constant having a controlled rate of divergence as $q\to2$, then a bound similar to \eqref{eq: smoothing-hyp} holds, but the upper bound is given in terms of the quantity
\begin{equation}\label{log-a}
[(\log t)^a/t]^{1/(m-1)}\ \ \ \forall t\ge2 \, ,
\end{equation}
where the constant $a$ is explicit and \it larger \rm than one. Thus, the resulting smoothing effect is given in terms of a slightly slower function of $t$, but the assumption is considerably weakened: in fact, we do not require that there exists a spectral gap for $-\Delta$ on the manifold under consideration. This result is contained in Theorem \ref{thm: sobo-deg}. We then construct simple explicit examples (in the class of the so-called \it model manifolds\rm, see e.g.\   \cite{Besse}, \cite{GW}, \cite{G}) in which the assumption of the theorem are satisfied, and the bounds of Theorem \ref{thm: sobo-deg} are \it sharp \rm in the sense that a matching lower bound holds for suitable solutions. This is shown by means of carefully constructed barriers, see Proposition \ref{thm: example}.

It is worth mentioning that similar bounds involving quantities like \eqref{log-a} have been obtained in \cite[Section 1.4]{A} for solutions to the homogeneous Dirichlet problem on Euclidean quasi-cylindrical domains. The techniques used there are completely different and rely on estimates of the growth of the support of the solution.

\smallskip
In Section \ref{sub} we concentrate on the case in which, besides the usual Sobolev inequality that we keep on assuming, an additional sub-Poincar\'e inequality holds. Namely, we require that there exists $p\in[1,2)$ such that
\begin{equation}\label{subsub}
\left\| f \right\|_{p;\nu} \le D \left\| \nabla f \right\|_{2;\mu} \quad \forall f \in C^\infty_c(M)
\end{equation}
holds for a suitable positive constant $D$, where the norms are taken w.r.t.\ suitable weights $\nu,\mu$, see  Section \ref{sub} for details. We show a surprising phenomenon, namely that under the stated assumptions the \it absolute bound \rm
\begin{equation}\label{abs}
\left\| u(t) \right\|_\infty \le B \, t^{-\frac{1}{m-1}} \quad \forall t>0
\end{equation}
holds for solutions to a weighted porous medium equation naturally related to the weights $\nu,\mu$. Such bound, when it holds, is clearly sharp, see Remark \ref{sharp}. As already mentioned, this is a bound which is typically known to hold in \it bounded \rm Euclidean domains; here we also show that there are examples of weights which make the weighted manifold considered of \it infinite volume\rm, but on which the assumptions hold and hence \eqref{abs} holds. To the best of our knowledge, an absolute bound in infinite-volume weighted manifolds first appeared in \cite{T} (see Theorem 1.3 there), where the author proves it in the special case of the Euclidean space with power-type weights, in the more general context of doubly nonlinear equations. In agreement with our results, one can easily verify that the weights considered there satisfy a sub-Poincar\'e inequality.

As a corollary, existence of a bounded, positive, minimal solution to a related weighted sublinear elliptic equation on manifolds is proved under the above assumptions.

Finally, we prove that the above sub-Poincar\'e inequality \eqref{subsub} \it cannot hold on Cartan-Hadamard manifolds\rm, when $\mu,\nu=1$, a purely functional analytic result which we believe has some independent interest, see Section \ref{sec-no-sub}.

\subsection{Existence, uniqueness and short-time estimates}
Existence of solutions to \eqref{eq-non-pesata}, which satisfy a suitable weak formulation, can be established by means of a rather standard approximation procedure. The main idea is to first solve the following homogeneous Dirichlet problem:
\begin{equation*}\label{eq-non-pesata-n}
\begin{cases}
u_t =  \Delta (u^m) & \textrm{in } B_n \times \mathbb{R}^{+} \, , \\
u=0 & \textrm{on } \partial B_n \times \mathbb{R}^{+} \, , \\
u = u_0 \in L^1(M) \cap L^\infty(M)  & \textrm{on } B_n \times \{ 0 \} \, ,
\end{cases}
\end{equation*}
where $B_n$ is a ball of radius $ n $, and then let $ n \to \infty $ using suitable energy estimates. This is performed in detail in \cite{GMP}, following a strategy outlined in \cite{V07}, in the case of \it weighted \rm porous medium equations in the Euclidean setting. The modifications required to cover the present situation are inessential. The same comment applies to the weighted porous medium equation dealt with in Section \ref{sub}. We also note that the assumption $u_0 \in L^\infty(M)$ can, likewise, be removed quite standardly.

Uniqueness is established e.g.\ in the class of weak \emph{energy} solutions corresponding to $ u_0 \in L^1(M)\cap L^{m+1}(M) $ (see again \cite{GMP}) or in the class of \emph{strong} solutions with $ u_0 \in L^1(M) $ (see \cite{V07}).

The smoothing effects proved in the present paper are relevant for large times only, although they hold for any $t>0$. In fact, we comment here that by a well-known result (see e.g.\  \cite{BG}) the validity of a Sobolev inequality of the type
\begin{equation*}\label{eq.Sob-fam-segnato2}
\| f \|_{2\sigma} \le C \, \|\nabla f \|_2 \quad \forall f\in C_c^\infty(M) \,
\end{equation*}
where $\sigma\in(1,d/(d-2)]$, implies that the following smoothing effect holds:
\begin{equation*}\label{eq.smoothing2}
\left\| u(t) \right\|_{\infty} \leq    K \frac{\|u_0\|_1^{\frac{\sigma-1}{\sigma m -1}}}{t^{\frac{\sigma}{\sigma m -1}}}\qquad\ \forall t>0,\ \ \forall u_0\in\,L^{1}(M)\, .
\end{equation*}
It is easy to see that the above bound, for any given $\|u_0\|_1$, is stronger than the ones proved in the present paper as $t\to0$ whatever the value of $\sigma$ is, whereas it is weaker as $t\to+\infty$.

% dire due parole anche sul problema pesato

% cambiare titolo
\section{Poincar\'e inequalities}\label{SE}
%\subsection{Poincar\'e and Sobolev-type inequalities}
We consider here a Cartan-Hadamard manifold for which $\min \Sigma_{{\rm L}^2}(-\Delta)=\Lambda>0$, where $\Sigma_{{\rm L}^2}(-\Delta)$ denotes the $L^2$ spectrum of $-\Delta$. This can be equivalently rewritten as
\begin{equation}\label{eq.Poi}
\sqrt \Lambda \, \|f\|_{2} \le  \|\nabla f \|_2 \quad \forall f \in C_c^\infty(M)
\end{equation}
for some $ \Lambda>0 $, namely a \emph{Poincar\'e} inequality. Note that, by the well-known results of \cite{M}, inequality \eqref{eq.Poi} holds if in addition the sectional curvatures are bounded above by a negative constant. On the other hand, on any Cartan-Hadamard manifold the Nash inequality
\begin{equation}\label{eq.Nash}
\| f \|_{2} \le C \, \| \nabla f \|_2^{ \frac{d}{d+2} } \| f \|_1^{\frac{2}{d+2}} \ \ \ \forall f \in C_c^\infty(M)
\end{equation}
holds (see e.g~\cite[Remark 14.6 and Lemma 14.7]{G-libro}) for some positive $ C=C(d) $. As a matter of fact, in dimension $ d \ge 3 $, the Nash inequality is actually equivalent to the Sobolev inequality
\begin{equation}\label{eq.Sob}
\|f\|_{2^*} \le C \, \|\nabla f\|_2 \ \ \ \forall f \in C_c^\infty(M) \, ,
\end{equation}
where $2^*:= 2d/(d-2)$ (we refer e.g~to \cite[Section 2.4]{Dav} or to \cite{BCLS}). Actually \eqref{eq.Nash} and \eqref{eq.Sob} are also equivalent to suitable Faber-Krahn inequalities, see \cite[Proposition 8.1]{H} or \cite[Lemma 14.7 and Exercise 14.3]{G-libro}, or to bounds on the volume growth of the level sets of the Green's function \cite[Theorem 8.2]{H} (such results basically rely on the papers \cite{HS,C,CL}).
% eventualmente dire che tutto vale su varieta complete dove basta chiedere sobolev e poincare? Approx del problema con palle non funziona?
%By \cite{H}, Prop. 8.1 and Theorems 8.2-8.4 (which rely on the results of \cite{HS,C,CL}), \eqref{eq.Sob} is equivalent to Faber-Krahn inequalities and to bounds on the growth of the Green's function.

We are now ready to state and prove the main result of this section.

\begin{thm}\label{teo: log-1}
Let $M$ be a Cartan-Hadamard manifold such that the Poincar\'e inequality \eqref{eq.Poi} holds. Then, for all solutions of the porous medium equation \eqref{eq-non-pesata}, the smoothing effect
\beq\label{eq.smoothgen}
\|u(t)\|_\infty\le H \left[\frac{\log\left(2+t\|u_0\|^{m-1}_1\right)}t\right]^{\frac1{m-1}} \quad \forall t>0
\eeq
holds, where $ H $ is a positive constant depending only on $ d $, $ \Lambda $ and $m$.

In particular, the statement is true on any Cartan-Hadamard manifold on which the condition
\[
\operatorname{sec}(x)\le-k<0
\]
holds for all $x\in M$ and a positive constant $k$, where $\operatorname{sec}$ denotes sectional curvature.
\end{thm}
\begin{proof}
Let us first suppose $ d \ge 3 $ and assume with no loss of generality that $u_0\not\equiv0$, $u_0\in L^\infty(M)$, such assumption being removable by standards methods after \eqref{eq.smoothgen} has been proved under such hypothesis. By using the methods given e.g.\  in \cite{BG}, \cite{GM}, it is well known that
\begin{equation}\label{eq.smoothing}
\left\| u(t) \right\|_{\infty} \leq  K t^{-\frac{\sigma}{(\sigma-1)(q+1)+\sigma(m-1)}} \left\|  u_0 \right\|_{q+1}^{\frac{(\sigma-1)(q+1)}{(\sigma-1)(q+1)+\sigma(m-1)}}  \ \ \ \forall t>0 \, , \ \forall q\ge0 \, , \ \forall u_0 \in L^{q+1}(M) \, ,
\end{equation}
where $\sigma:=d/(d-2)$. It will important to note that $K$ depends only on $d, m$ and on the constant $C$ appearing in \eqref{eq.Sob}, but can be taken to be independent of $q$.

We now make use of the Poincar\'e-type inequality \eqref{eq.Poi} to give some further bounds on $\|u(t)\|_q$ for $q<+\infty$. We proceed as follows, denoting by ${\rm d}x$ the Riemannian measure:
\begin{equation}\label{eq.deriv}
\begin{aligned}
\frac{{\rm d}}{{\rm d}t}\|u(t)\|_{q+1}^{q+1}&=(q+1)\int_M u^q(t)\Delta u^m(t)\,{\rm d}x \\ &= -q(q+1)m\int_M u^{q+m-2}(t)|\nabla u(t)|^2\,{\rm d}x\\
&=-\frac{4q(q+1)m}{(m+q)^2}\int_M |\nabla u^{\frac{q+m}2}(t)|^2\,{\rm d}x\\
&\le -\frac{4q(q+1)m\,\Lambda}{(m+q)^2}\int_M u^{q+m}(t)\,{\rm d}x,
\end{aligned}
\end{equation}
where $\Lambda$ is the constant appearing in \eqref{eq.Poi}. We now use the interpolation inequality
\[
\|f\|_{q+1}\le\|f\|_{q+m}^\vartheta\|f\|_1^{1-\vartheta},
\]
where
\[\vartheta=\frac{q(q+m)}{(q+1)(q+m-1)}\in(0,1)\]
since $m>1$. Rewrite this as
\[
\|f\|_{q+m}^{q+m}\ge\frac{\|f\|_{q+1}^{\frac{q+m}{\vartheta}}}{\|f\|_1^{\left(\frac1\vartheta-1\right)(q+m)}}
\]
and use it in \eqref{eq.deriv} to get
\begin{equation*}\label{eq.deriv2}
\begin{aligned}
\frac{{\rm d}}{{\rm d}t}\|u(t)\|_{q+1}^{q+1} &\le -\frac{4q(q+1)m\,\Lambda}{(m+q)^2}\int_M u^{q+m}(t)\,{\rm d}x\\ &\le -\frac{4q(q+1)m\,\Lambda}{(m+q)^2}\frac{\|u(t)\|_{q+1}^{\frac{q+m}{\vartheta}}}{\|u(t)\|_1^{\left(\frac1\vartheta-1\right)(q+m)}}\\
&\le-\frac{4q(q+1)m\,\Lambda}{(m+q)^2}\frac{\|u(t)\|_{q+1}^{\frac{q+m}{\vartheta}}}{\|u_0\|_1^{\left(\frac1\vartheta-1\right)(q+m)}} \, ,
\end{aligned}
\end{equation*}
where in the last step we have used non-expansivity of the $L^1$ norm. Hence, for initial data $u_0$ such that $\|u_0\|_1=1$ we have shown that
\[
\frac{{\rm d}}{{\rm d}t}\|u(t)\|_{q+1}^{q+1}\le -\frac{4q(q+1)m\,\Lambda}{(m+q)^2}\left(\|u(t)\|_{q+1}^{q+1}\right)^\lambda,
\]
where $\lambda=1+(m-1)/q>1$. Integrating in time we get, setting $y(t):=\|u(t)\|_{q+1}^{q+1}$,
\[
\frac1{y(t)^{\lambda-1}}\ge\frac1{y(0)^{\lambda-1}}+\frac{4q(q+1)m(\lambda-1)\,\Lambda}{(m+q)^2}t>\frac{4(q+1)m(m-1)\,\Lambda}{(m+q)^2}t,
\]
so that
\begin{equation}\label{eq.abs}
\|u(t)\|_{q+1}\le \left(\frac{(m+q)^2}{4(q+1)m(m-1)\,\Lambda}\right)^{\frac q{(m-1)(q+1)}}\frac1{t^{\frac q{(m-1)(q+1)}}}.
\end{equation}
Combining \eqref{eq.smoothing}, written in the time interval $[t/2,t]$ and \eqref{eq.abs}, written in the time integral $[0,t/2]$, yields:
\begin{equation*}
\begin{aligned}
\|u(t)\|_\infty\le & K \left(\frac2t\right)^{\frac{\sigma}{(\sigma-1)(q+1)+\sigma(m-1)}} \left\|  u\left(t/2\right) \right\|_{q+1}^{\frac{(\sigma-1)(q+1)}{(\sigma-1)(q+1)+\sigma(m-1)}}\\
\le & K \left(\frac2t\right)^{\frac{\sigma}{(\sigma-1)(q+1)+\sigma(m-1)}}\left(\frac{(m+q)^2}{4(q+1)m(m-1)\,\Lambda}\right)^{ \frac{q(\sigma-1)}{(m-1)[(\sigma-1)(q+1)+\sigma(m-1)]}} \\
& \times \left(\frac2t\right)^{ \frac{q(\sigma-1)}{(m-1)[(\sigma-1)(q+1)+\sigma(m-1)]}}\\
= & K\left(\frac2t\right)^{ \frac{\sigma(m-1)+q(\sigma-1)}{(m-1)[(\sigma-1)(q+1)+\sigma(m-1)]}}\left(\frac{(m+q)^2}{4(q+1)m(m-1)\,\Lambda}\right)^{ \frac{q(\sigma-1)}{(m-1)[(\sigma-1)(q+1)+\sigma(m-1)]}}\\
= & K\frac{2^{\frac{\sigma(m-1)+q(\sigma-1)}{(m-1)[(\sigma-1)(q+1)+\sigma(m-1)]}}}{t^{\frac1{m-1}-\frac{\sigma-1}{(m-1)[(\sigma-1)(q+1)+\sigma(m-1)]}}}
\left(\frac{(m+q)^2}{4(q+1)m(m-1)\,\Lambda}\right)^{\frac1{m-1}- \frac{\sigma-1+\sigma(m-1)}{(m-1)[(\sigma-1)(q+1)+\sigma(m-1)]}}\\
= & K \left(\frac{(m+q)^2}{4(q+1)m(m-1)\,\Lambda\,t}\right)^{\frac1{m-1} }t^{\frac{\sigma-1}{(m-1)[(\sigma-1)(q+1)+\sigma(m-1)]}}\\
& \times 2^{\frac{\sigma(m-1)+q(\sigma-1)}{(m-1)[(\sigma-1)(q+1)+\sigma(m-1)]}}\left(\frac{4(q+1)m(m-1)\,\Lambda}{(m+q)^2}
\right)^{\frac{\sigma-1+\sigma(m-1)}{(m-1)[(\sigma-1)(q+1)+\sigma(m-1)]}}\\
\le & K'\left(\frac{(m+q)^2}{4(q+1)m(m-1)\,\Lambda\,t}\right)^{\frac1{m-1} }t^{\frac{\sigma-1}{(m-1)[(\sigma-1)(q+1)+\sigma(m-1)]}},
\end{aligned}
\end{equation*}
where in the last step we have noticed that the quantity
\[
2^{\frac{\sigma(m-1)+q(\sigma-1)}{(m-1)[(\sigma-1)(q+1)+\sigma(m-1)]}}\left(\frac{4(q+1)m(m-1)\,\Lambda}{(m+q)^2}
\right)^{\frac{\sigma-1+\sigma(m-1)}{(m-1)[(\sigma-1)(q+1)+\sigma(m-1)]}}
\]
is uniformly bounded as a function of $q\ge0$, and we have denoted by $K'$ another numerical constant which, as $K$ does, depends only on $d, m$, on the constant $C$ appearing in \eqref{eq.Sob} and on the constant $\Lambda$ appearing in \eqref{eq.Poi}, but can be taken to be independent of $q$.

Hence we have shown that
\begin{equation}\label{eq.q}
\|u(t)\|_\infty\le K'\left(\frac{(m+q)^2}{4(q+1)m(m-1)\,\Lambda\,t}\right)^{\frac1{m-1} }t^{\frac{\sigma-1}{(m-1)[(\sigma-1)(q+1)+\sigma(m-1)]}}.
\end{equation}

Choose now, for $t\ge0$, $q=\log(t+2)$. Making this choice in \eqref{eq.q} yields
\[
\|u(t)\|_\infty\le H\left[\frac{\log(t+2)}t\right]^{\frac1{m-1}},
\]
for all $u_0$ such that $\|u_0\|_1=1$, where we have noticed that $t^{1/\log t}=e$ and where $H$ is a constant depending again on $d,m,C,\Lambda$. This is the stated bound if $\|u_0\|_1=1$. The claim follows by a standard scaling argument in the time variable.

Finally, if $ d = 2 $, the smoothing effect
\begin{equation}\label{eq.smoothing-2d}
\left\| u(t) \right\|_{\infty} \leq  K \, t^{-\frac{1}{q+m}} \left\|  u_0 \right\|_{q+1}^{\frac{q+1}{q+m}} \ \ \ \forall t>0 \, , \ \forall q\ge0 \, , \ \forall u_0 \in L^{q+1}(M) \, ,
\end{equation}
can be shown to hold (one can use the techniques e.g.\  of \cite{GM3}). The proof of \eqref{eq.smoothgen} then follows as above, just by using \eqref{eq.smoothing-2d} instead of \eqref{eq.smoothing}.
\end{proof}
\begin{rem}\label{oss: hyp}\rm
The bound given in the previous theorem is sharp. In fact it coincides with the smoothing effect recently proved, in the case of $M={\mathbb H}^d$ (the $d$-dimensional hyperbolic space, in which sectional curvatures are everywhere equal to -1), in \cite{V}, mainly for radial data. In \cite{V}, the result is proved by delicate barrier arguments, specific for the case of ${\mathbb H}^d$, and is shown to be sharp. Our contribution here is to link the smoothing effect to the validity of suitable functional inequalities only and, hence, to curvature assumptions.
\end{rem}

\begin{rem} \rm The recent paper \cite{AT} proves new interesting bounds for solutions to the doubly nonlinear equation on Riemannian manifolds, which are sharp in some cases. See also \cite{AT2, AT3} for similar results in the case of unbounded Euclidean domains.

The authors use an approach based on Faber-Krahn inequalities, which is known to be fruitful in the linear case, see e.g.\ \cite{G2, G3}. In such paper the key assumption concerns a suitable control of $\Lambda(v)$, the bottom of the spectrum of the Laplacian with additional Dirichlet boundary conditions on the boundary of a subset $A$, in terms of the volume $v$ of $A$, see hypothesis H1 of that paper. It is easily shown that, however, their assumption H1 does not hold in ${\mathbb H}^d$, although it might be possible that their method of proof can be adapted to describe the case of negative curvature as well. We also comment that, even if the result given in Remark 1.4 of \cite{AT} should hold, it would yield a bound of the form
\begin{equation*}\label{eq: smoothing-hyp2}
\|u(t)\|_\infty\le C\left[\frac{(\log t)^2}t\right]^{\frac1{m-1}}\ \ \ \forall t\ge2\, ,
\end{equation*}
when $\|u_0\|_1=1$, which is weaker than the one proved in \cite{V}, and here in the general case. Similar comments apply to the examples given in the next section.
\end{rem}

\section{Sobolev-type inequalities diverging as $ \sigma \downarrow 1 $}\label{deg}

Suppose that the Sobolev-type inequality
\begin{equation*}\label{eq.Sob-fam-segnato}
\| f \|_{2\overline{\sigma}} \le C \, \|\nabla f \|_2 \quad \forall f\in C_c^{\infty}(M) \,
\end{equation*}
holds for a suitable $\overline{\sigma}>1$. Assume in addition that the Poincar\'e inequality \eqref{eq.Poi} holds. Then for any $ \sigma \in (1,\overline{\sigma}) $ the inequality
\begin{equation}\label{eq.Sob-fam}
\| f \|_{2\sigma} \le C_\sigma \, \|\nabla f \|_2 \quad \forall f\in C_c^{\infty}(M) \, , \ \forall \sigma \in (1,\overline{\sigma}] \, ,
\end{equation}
holds as well, where the positive constant $ C_\sigma $ depends on $ \Lambda $, $C$ and can actually taken to be independent of $ \sigma \in (1,\overline{\sigma}) $. The aim of this section is to provide asymptotic bounds for the solutions to \eqref{eq-non-pesata} when the family \eqref{eq.Sob-fam} holds on $M$ with a constant $ C_\sigma $ that may blow up as $ \sigma \downarrow 1 $, with a suitable rate. We state below the result of the present section. Later on we shall concentrate on specific examples.
\begin{thm}\label{thm: sobo-deg}
Suppose that the family of Sobolev-type inequalities \eqref{eq.Sob-fam} holds, with
\begin{equation}\label{eq: cond-C}
C_\sigma \le \frac{\widehat{C}}{(\sigma-1)^\gamma} \quad \forall \sigma \in (1,\overline{\sigma}]
\end{equation}
for some $ \gamma>0 $ and $ \widehat{C}>0 $. Then there exists a constant $ Q>0 $, depending only on $ \overline{\sigma} $, $ C_{\overline{\sigma}} $, $ \gamma$, $ \widehat{C} $ and $m$, such that any solution $u$ to \eqref{eq-non-pesata} corresponding to an initial datum $u_0\in L^1(M)$ satisfies the smoothing estimate
\begin{equation}\label{eq: log-alfa}
\left\| u(t) \right\|_\infty \le Q \left[ \log\left(t \, \| u_0 \|_1^{m-1} + e \right) \right]^{\frac{1+2\gamma}{m-1}}  t^{-\frac{1}{m-1}}  \quad \forall t>0 \, .
\end{equation}
\end{thm}

\begin{rem}\rm
If $\gamma=0$ we recover the result of the previous section, whose setting is however geometrically clearer and for which we then preferred to give an independent and simpler proof.
\end{rem}
\begin{proof}
Let $ q >0  $ and $ \sigma \in (1,\overline{\sigma} ]$. With no loss of generality, we shall suppose that $ u_0 \in L^1(M) \cap L^\infty(M) $. Proceeding as in the proof of Theorem \ref{teo: log-1} and exploiting \eqref{eq.Sob-fam}, we get:
\begin{equation}\label{eq:log-gamma-1}
\frac{{\rm d}}{{\rm d}t}\|u(t)\|_{q+1}^{q+1} \le -\frac{4q(q+1)m}{(m+q)^2\,C_\sigma^2} \, \|u(t)\|_{\sigma(q+m)}^{q+m} \, .
\end{equation}
By interpolation and using the non-expansivity of the $ L^1 $ norm, we infer that
\begin{equation}\label{eq:log-gamma-2}
\| u(t) \|_{q+1} \le \| u(t) \|_{\sigma(q+m)}^{\frac{\sigma(q+m)q}{[\sigma(q+m)-1](q+1)}} \, \| u_0 \|_{1}^{\frac{\sigma(q+m)-(q+1)}{[\sigma(q+m)-1](q+1)}} \quad \forall t>0 \, .
\end{equation}
Assuming for the sake of notational simplicity that $ \| u_0 \|_1=1 $, from \eqref{eq:log-gamma-1} and \eqref{eq:log-gamma-2} it follows
\begin{equation}\label{eq:log-gamma-3}
\frac{{\rm d}}{{\rm d}t}\|u(t)\|_{q+1}^{q+1} \le -\frac{4q(q+1)m}{(m+q)^2\,C_\sigma^2} \, \|u(t)\|_{q+1}^{(q+1)\frac{\sigma(q+m)-1}{\sigma q}} \, .
\end{equation}
Integrating \eqref{eq:log-gamma-3} and setting $ y(t):=\| u(t) \|_{q+1}^{q+1} $ we obtain
$$ y(t)^{\frac{\sigma m -1}{\sigma q}} \le \frac{1}{\frac{1}{y(0)^{\frac{\sigma m -1}{\sigma q}}} + \frac{4m(q+1)(\sigma m -1)}{\sigma (q+m)^2 C_\sigma^2 } \, t } \quad \forall t>0 \, , $$
whence
\begin{equation}\label{eq:log-gamma-4}
\| u(t) \|_{q+1} \le \left[ \frac{\sigma (q+m)^2 C_\sigma^2 }{4m(q+1)(\sigma m -1)} \right]^{\frac{\sigma q}{(q+1)(\sigma m -1)}} t^{-\frac{\sigma q}{(q+1)(\sigma m -1)}} \quad \forall t>0 \, .
\end{equation}
Now we remark that the validity of \eqref{eq.Sob-fam} for $ \sigma=\overline{\sigma} $ implies the smoothing estimate
\begin{equation}\label{eq:log-gamma-5}
\left\| u(t) \right\|_{\infty} \leq  K \, t^{-\frac{\overline{\sigma}}{(\overline{\sigma}-1)(q+1)+\overline{\sigma}(m-1)}} \left\|  u_0 \right\|_{q+1}^{\frac{(\overline{\sigma}-1)(q+1)}{(\overline{\sigma}-1)(q+1)+\overline{\sigma}(m-1)}} \quad \forall t>0
\end{equation}
for some positive $K=K(\overline{\sigma},C_{\overline{\sigma}},m)$, this following e.g.\ from the results of \cite{GM}. Hence by combining \eqref{eq:log-gamma-4} (at time $t/2$) and \eqref{eq:log-gamma-5} (with the time origin shifted to $t/2$) we end up with
\begin{equation}\label{eq:log-gamma-6}\begin{aligned}
\left\| u(t) \right\|_{\infty} \leq  A &\left[ \frac{\sigma (q+m)^2 C_\sigma^2 }{4m(q+1)(\sigma m -1)} \right]^{\frac{\sigma q (\overline{\sigma}-1)}{ (\sigma m -1)[(\overline{\sigma}-1)(q+1)+\overline{\sigma}(m-1) ] }}\\ &\times t^{-\frac{\overline{\sigma}(\sigma m -1)+\sigma q (\overline{\sigma}-1)}{ (\sigma m -1)[(\overline{\sigma}-1)(q+1)+\overline{\sigma}(m-1) ] }}  \quad \forall t>0 \, ,\end{aligned}
\end{equation}
where $ A $ is a suitable positive constant that can be taken to depend only on $ \overline{\sigma} $, $C_{\overline{\sigma}}$ and $m$. At this point, as in the proof of Theorem \ref{teo: log-1}, we can let $ q=\log(t+e) $ in \eqref{eq:log-gamma-6} to get
\begin{equation}\label{eq:log-gamma-7}
\begin{aligned}
\left\| u(t) \right\|_{\infty} \leq & A \, \left\{ \frac{\sigma [\log(t+e)+m]^2}{4m[\log(t+e)+1](\sigma m -1)} \right\}^{-\frac{\sigma(\overline{\sigma}m-1)}{(\sigma m-1)\{ (\overline{\sigma}-1)[\log(t+e)+1] + \overline{\sigma}(m-1) \}}}  \\
& \times t^{\frac{\overline{\sigma}-\sigma}{(\sigma m-1)\{(\overline{\sigma}-1)[\log(t+e)+1]+\overline{\sigma}(m-1)\}}} \left\{ \frac{\sigma [1+m/\log(t+e)]^2}{4m[1+1/\log(t+e)](\sigma m -1)} \right\}^{\frac{\sigma}{\sigma m-1}} \\
& \times C_\sigma^{-\frac{2\sigma(\overline{\sigma}m-1)}{(\sigma m-1)\{ (\overline{\sigma}-1)[\log(t+e)+1] + \overline{\sigma}(m-1) \}}} \left[ \log(t+e) \, C_\sigma^2 \right]^{\frac{\sigma}{\sigma m-1}} t^{-\frac{\sigma}{\sigma m-1}} \quad \forall t>0
\end{aligned}
\end{equation}
Since $ \sigma \in (1,\overline{\sigma}) $, it is apparent that the first two factors in the r.h.s.~of \eqref{eq:log-gamma-7} can be bounded from above by another positive constant $ A^\prime $ that depends only on $ \overline{\sigma} $, $C_{\overline{\sigma}}$ and $m$, so that \eqref{eq:log-gamma-7} reads
\begin{equation*}\label{eq:log-gamma-8}
\left\| u(t) \right\|_{\infty} \leq  A^\prime \, C_\sigma^{-\frac{2\sigma(\overline{\sigma}m-1)}{(\sigma m-1)\{ (\overline{\sigma}-1)[\log(t+e)+1] + \overline{\sigma}(m-1) \}}} \left[ \log(t+e) \, C_\sigma^2 \right]^{\frac{\sigma}{\sigma m-1}} t^{-\frac{\sigma}{\sigma m-1}} \quad \forall t>0 \, ,
\end{equation*}
which implies, in view of \eqref{eq: cond-C},
\begin{equation*}\label{eq:log-gamma-9}\begin{aligned}
\left\| u(t) \right\|_{\infty} \leq  &A^{\prime\prime} \, (\sigma-1)^{\frac{2\gamma \sigma(\overline{\sigma}m-1)}{(\sigma m-1)\{ (\overline{\sigma}-1)[\log(t+e)+1] + \overline{\sigma}(m-1) \}}}\\ &\times\left[ \log(t+e) \left(\sigma-1\right)^{-2\gamma} \right]^{\frac{\sigma}{\sigma m-1}} t^{-\frac{\sigma}{\sigma m-1}} \quad \forall t>0 \, ,\end{aligned}
\end{equation*}
where $ A^{\prime\prime} $ is a suitable positive constant depending on the same parameters on which $ A^\prime $ depends and on $ \widehat{C} $. We define
$$ \sigma = 1 + \frac{\overline{\sigma}-1}{\log(t+e)} \, , $$
whence, up to another positive constant $ A^{\prime\prime\prime}=A^{\prime\prime\prime}(\overline{\sigma}, C_{\overline{\sigma}}, \gamma, \widehat{C} , m) $,
\begin{equation*}\label{eq:log-gamma-10}\begin{aligned}
\left\| u(t) \right\|_{\infty} \leq  &A^{\prime\prime\prime} \left[ \log(t+e) \right]^{\frac{1+2\gamma}{m-1} - \frac{(1+2\gamma)(\overline{\sigma}-1)}{[(m-1)\log(t+e)+m(\overline{\sigma}-1) ](m-1)} }\\ &\times t^{-\frac{1}{m-1} + \frac{\overline{\sigma}-1}{[(m-1)\log(t+e)+m(\overline{\sigma}-1) ](m-1)}} \quad \forall t>0 \, ,
\end{aligned}\end{equation*}
that is
\begin{equation*}\label{eq:log-gamma-11}
\left\| u(t) \right\|_{\infty} \leq  Q \left[ \log(t+e) \right]^{\frac{1+2\gamma}{m-1}} t^{-\frac{1}{m-1} } \quad \forall t>0 \, ,
\end{equation*}
where $ Q $ is the final constant of the statement. The validity of \eqref{eq: log-alfa} follows by scaling, as in the end of the proof of Theorem \ref{teo: log-1}.
\end{proof}

\subsection{Examples}\label{exa-3}
We shall discuss here some particular manifolds on which the assumptions of Theorem \ref{thm: sobo-deg} hold, and on which the bound proved in such theorem can be shown to be sharp.

The manifolds we have in mind are particular \it model manifolds\rm, namely they are topologically ${\mathbb R}^d$ with their metric being given by
\begin{equation*}\label{eq-modello}
{\rm d}s^2={\rm d}r^2+\psi(r)^2{\rm d}\Theta^2,
\end{equation*}
where $r$ denotes geodesic distance from a given \it pole \rm $o$, ${\rm d}\Theta^2$ denotes the canonical metric on the unit sphere ${\mathbb S}^{d-1}$ and $\psi:[0,+\infty)\longrightarrow[0,+\infty)$ is a $C^2$ function with $\psi(0)=\psi_+''(0)=0$ and $\psi'_+(0)=1$, where the subscript ``+'' denotes right derivative. Note that these conditions imply that the metric is $C^2$ and can be extended across $o$. We also assume throughout that $\psi'(r)>0$ for all $r>0$.

In the above coordinates, we have that
\begin{equation*}\label{grad-mod}\begin{aligned}
\nabla f (r,\phi_1, \ldots , \phi_{d-1}) &=  \frac{\partial f}{\partial r}(r,\phi_1,\ldots,\phi_{d-1}) \, \mathbf{e}_r\\ &+ \frac{1}{\psi(r)} \sum_{n=1}^{d-1} \frac{\partial f}{\partial \phi_n}(r,\phi_1,\ldots,\phi_{d-1}) \,  \mathbf{e}_{\phi_n} \, ,
\end{aligned}\end{equation*}
where $ \mathbf{e}_r $ is the radial versor and $ \{ \mathbf{e}_{\phi_n} \}_{n=1\ldots d-1} $ is a canonical basis for the tangent space of the sphere; moreover,
\[\begin{aligned}
\Delta f(r,\phi_1,\ldots,\phi_{d-1})= & \frac1{(\psi(r))^{d-1}}\frac{\partial}{\partial
r}\left[(\psi(r))^{d-1} \frac{\partial f}{\partial
r}(r,\phi_1,\ldots,\phi_{d-1})
\right]\\
& + \frac1{(\psi(r))^2}\Delta_{{\mathbb
S}^{d-1}}f(r,\phi_1,\ldots,\phi_{d-1}) \, .
\end{aligned}
\]
For \it radial \rm functions, that is functions depending on $r$ only, one then has
\begin{equation}\label{lap-mod}
\Delta f(r)=f^{\prime\prime}(r)+(d-1)\frac{\psi^\prime(r)}{\psi(r)}f^\prime(r) \, .
\end{equation}
We underline the following well-known geometrical facts, for which we refer e.g.\ to \cite{Besse}, \cite{GW}, \cite{G}:
\begin{itemize}
\item the quantity $(n-1)\frac{\psi^\prime(r)}{\psi(r)}$ represents the mean curvature of the geodesic sphere of radius $r$ in the radial direction;
\item let $\omega_d$ be the volume of the $d$-dimensional unit sphere. Then
\[
S(r):=\omega_n (\psi(r))^{d-1},\ \ \ V(r):=\int_0^rS(t)\,{\rm d}t=\omega_n\int_0^r(\psi(t))^{d-1}\,{\rm d}t
\]
are the area of the geodesic sphere $\partial B_r$ and the volume of the geodesic ball $B_r(o)$, respectively;
\item let $\rm Ric\it\, (\partial r,\partial r)$ be the Ricci tensor in the radial direction, and $K_\pi(r)$ be
sectional curvature w.r.t planes containing $\partial r$. Then
\[
\frac1{n-1}\rm Ric\it\, (\partial r,\partial
r)=K_\pi(r)=-\frac{\psi^{\prime\prime}(r)}{\psi(r)}.
\]
Moreover, the sectional curvature w.r.t.\ planes orthogonal to $\partial r$ is given by $\frac{1-(\psi^\prime(r))^2}{(\psi(r))^2}$. Sectional curvatures equal -1 on the hyperbolic space, whereas they are still negative, but tending to zero (as multiples of $-r^{2(a-1)}$) when one has, for large $r$, $\psi(r)=e^{r^a}$ for some $a\in(0,1)$, which is the case we are going to discuss in more detail hereafter.
\end{itemize}

As just said, we concentrate on the case in which $\psi(r)\sim c_1e^{c_2r^a/(d-1)}$ for large $r$ and for some given $a\in(0,1)$, $c_1,c_2>0$. Such manifolds, loosely speaking, interpolate between the Euclidean situation (where $\psi(r)=r$) and the Hyperbolic situation (where $\psi(r)=\sinh r\sim e^r/2$ as $r\to+\infty$). We perform the following calculations in the case $c_1=c_2=1$ for notational simplicity, assuming also again for simplicity that $\psi(r)= e^{r^a/(d-1)}$ for $r\ge1$. Very similar calculations can be performed in the assumptions of Proposition \ref{thm: example}.

\smallskip

Under such assumptions, we point out that \eqref{lap-mod} holds with $ (d-1)\psi^\prime(r)/\psi(r) = h(r)  $, where
$$ b_1 \, \frac{1}{r} \le h(r) \le b_2 \, \frac{1}{r} \ \ \forall r \in (0,1) \, , \quad D_1 \, \frac{a}{r^{1-a}} \le h(r) \le D_2 \, \frac{a}{r^{1-a}} \ \ \forall r \ge 1 \, ,  $$
for some positive constants $ D_1 $ and $ D_2 $ depending on $ c_1 $, $ c_2 $ and $d$, and $ b_1 $, $b_2$ depending on $ c_1 $, $c_2$, $d$ and $a$.

\smallskip

Our first goal will be to show that, for the manifold considered, the Sobolev inequality \eqref{eq.Sob-fam} holds with $\overline{\sigma}=d/(d-2)$, and with condition \eqref{eq: cond-C} holding as well for a suitable $\gamma>0$, \it provided radial functions are considered\rm.

To this end we use the characterization of one-dimensional Sobolev inequality given in \cite[Theorem 1.14]{kufner}. In our case, the condition to be satisfied is the following:
\begin{equation*}\label{eq.sobolev-1d}
C_\sigma:=\sup_{x>0}\left(\int_0^x[\psi(r)]^{d-1}{\rm d}r\right)^{\frac1{2\sigma}}\left(\int_x^{+\infty}[\psi(r)]^{1-d}{\rm d}r\right)^{\frac12}<+\infty,
\end{equation*}
and $C_\sigma$ given above can then be taken as the constant appearing in \eqref{eq.Sob-fam}. It is easy to verify that the quantity
\[
\left(\int_0^x[\psi(r)]^{d-1}{\rm d}r\right)^{\frac1{2\sigma}}\left(\int_x^{+\infty}[\psi(r)]^{1-d}{\rm d}r\right)^{\frac12}
\]
is uniformly bounded as a function of $x\in(0,1]$ and $\sigma\in[1,d/(d-2)]$. So we are left with investigating the quantity
\[
A(x):=\left(\int_0^xe^{(d-1)r^a}{\rm d}r\right)^{\frac1{2\sigma}}\left(\int_x^{+\infty}e^{-(d-1)r^a}{\rm d}r\right)^{\frac12}
\]
for $x>1$ and $\sigma\in[1,d/(d-2)]$. Elementary calculations involving de l'H\^opital's rule show that
\[
A(x)\asymp x^{\frac{(1-a)(\sigma+1))}{2\sigma}}e^{-\frac{(d-1)(\sigma-1)}{2\sigma}x^a}\ \ \textrm{as}\ x\to+\infty,
\]
where the proportionality constants can be taken not to depend on $\sigma\in[1,d/(d-2)]$. Elementary estimates involving the function \[g(x):=x^{\frac{(1-a)(\sigma+1))}{2\sigma}}e^{-\frac{(d-1)(\sigma-1)}{2\sigma}x^a}\] show that
\[
A(x)\le \frac{\widehat C}{(\sigma-1)^{\frac{1-a}a}}\ \ \ \ \forall x\ge1
\]
where ${\widehat C}>0$ does not depend on $\sigma$, and $\sigma\in(1,d/(d-2)]$. This yields our claim for radial functions. We shall summarize this statement in Proposition \ref{thm: example} below.

Besides, we claim that in the above mentioned examples, the bound \eqref{eq: log-alfa} is sharp.

To this end, we shall construct a suitable (radial) subsolution. It will be of the following form:
\[
\underline{u}(r,t):=\begin{cases}C \, (t+t_0)^{-\frac1{m-1}}\left[\eta\,[\log(t+t_0)]^{\frac{2-a}{a}}-r^{2-a}\right]_+^{\frac1{m-1}} &\textrm{for}\ r\ge1,\\
C \,(t+t_0)^{-\frac1{m-1}}\left[\eta\,[\log(t+t_0)]^{\frac{2-a}{a}}-\frac{2-a}{2}r^{2}-\frac a2\right]_+^{\frac1{m-1}} &\textrm{for}\ r\in[0,1).
\end{cases}
\]
We omit the long, although straightforward, calculations leading to the proof that $\underline{u}$ is indeed a subsolution if $t_0$ is sufficiently large, $C$ and, subsequently, $\gamma$ are instead chosen to be small enough. Given this fact, it is easy to show that the time behaviour of $\|\underline{u}(t)\|_\infty$ is exactly the one predicted by \eqref{eq: log-alfa}, with $\gamma=(1-a)/a$.

We summarize the above results in the following proposition. Hereafter we use the subscript ``rad'' to indicate that functions in the corresponding space are assumed to be radial.

\begin{prop}\label{thm: example}
Let $M$ be a model manifold associated with a function $\psi$ satisfying $\psi'(r)>0$ for all $r>0$ and
\[\psi(r)\sim c_1 e^{c_2r^a},\ \ \psi'(r)\sim \left(c_1 e^{c_2r^a}\right)'\ \ \textrm{as}\ r\to+\infty \, ,\]
where $c_1,c_2$ are positive constants and $a\in(0,1)$.

Then the Sobolev inequalities \eqref{eq.Sob-fam}, restricted to \emph{radial functions}, hold for all $\sigma\in(1,d/(d-2))$. Moreover, the constant $C_\sigma$ in \eqref{eq.Sob-fam} satisfies condition \eqref{eq: cond-C} with $\gamma=(1-a)/a$. Hence, the bound \eqref{eq: log-alfa} holds for \emph{radial solutions} $u(t)$ of the porous medium equation posed on $M$ corresponding to data $u_0\in L^1_{\rm rad}(M)$, that is
\[
\|u(t)\|_\infty\le K \left[ \log\left(t \, \| u_0 \|_1^{m-1} + e \right) \right]^{\frac{2-a}{a(m-1)}}  t^{-\frac{1}{m-1}} \quad \forall t > 0 \, .
\]

In addition, such bound is sharp, in the sense that a matching lower bound can be given for an appropriate class of initial data.
\end{prop}

\begin{rem}\rm 
Note that, under the same assumptions on $M$ as in Proposition \ref{thm: example} and as a corollary of the latter, the bound $ \| u(t) \|_\infty \lesssim \left( \log t \right)^{\frac{2-a}{a(m-1)}}  t^{-\frac{1}{m-1}}  $ for large $t$ also holds for non-radial solutions having a compactly supported initial datum, this being a consequence of standard comparison arguments.
%The Sobolev inequalities \eqref{eq.Sob-fam} for the weights considered in the above proposition hold in the non-radial setting as well. However, in this case we can only show that the constant $ C_\sigma $ satisfies condition \eqref{eq: cond-C} with $\gamma=2/a$. This can be deduced from \cite[Section 2.1]{GSC} and \cite{CS} as a consequence of the growth of the volume of balls valid in the present case. This fact implies a decay rate for the corresponding solutions of \eqref{eq-non-pesata} that is not sharp in the radial case.
\end{rem}

\section{Sub-Poincar\'e inequalities}\label{sub}
We consider two positive weights $ \rho_\nu $ and $ \rho_\mu $ such that
\[
\rho_\nu, \rho_\nu^{-1},\rho_\mu, \rho_\mu^{-1}\in L_{{\rm loc}}^\infty(M).
\]
It will be required that the sub-Poincar\'e inequality
\begin{equation}\label{sub-poin}
\left\| f \right\|_{p;\nu} \le D \left\| \nabla f \right\|_{2;\mu} \quad \forall f \in C^\infty_c(M)
\end{equation}
holds for some $ p \in [1,2) $, $ D>0 $ and that the Sobolev-type inequality
\begin{equation}\label{eq.Sob-wei}
\| f \|_{2\sigma;\nu} \le C  \, \|\nabla f\|_{2;\mu}  \quad \forall f \in C^\infty_c(M)
\end{equation}
holds for some $ \sigma > 1 $, $ C>0 $. Here, for all $p\ge1$ we set
\[
L^p(M;\mu):=\left\{f: \|f\|_{p;\mu}:=\left(\int_M|f|^p\,{\rm d}\mu\right)^{1/p}<+\infty\right\},
\]
with a similar definition for $L^p(M;\nu)$.

Below we shall give examples of weights such that the above assumption hold. We also refer to \cite{GM} for larger classes of examples for which \eqref{eq.Sob-wei} holds, noting also that an analogous approach allows us to identify by similar methods, using the results of \cite{kufner}, weights such that also \eqref{sub-poin} is valid.

Under such assumptions we investigate the long-time behaviour of solutions to the following weighted porous medium equation:
\begin{equation}\label{eq-pesata}
\begin{cases}
\rho_\nu(x)u_t = \operatorname{div}\left[\rho_\mu(x) \nabla{(u^m)}\right] & \textrm{in } M \times \mathbb{R}^{+} \, , \\
u =u_0 \in L^1(M;\nu) & \textrm{on } M \times \{ 0 \} \, .
\end{cases}
\end{equation}

Our main result is the following.
\begin{thm}\label{thm: sub-poin}
Suppose that the sub-Poincar\'e inequality \eqref{sub-poin} and the Sobolev-type inequality \eqref{eq.Sob-wei} hold. Then there exists a constant $ B>0 $, depending on $p$, $ D $, $ \sigma $, $C$ and $m$ but independent of $u_0$, such that any solution $u$ to \eqref{eq-pesata} satisfies the absolute bound
\begin{equation}\label{eq: absb}
\left\| u(t) \right\|_\infty \le B \, t^{-\frac{1}{m-1}} \quad \forall t>0 \, .
\end{equation}
\end{thm}
\begin{proof}
With no loss of generality let us assume $ u_0 \not\equiv 0 $ and $ u_0 \in L^1(M;\nu) \cap L^\infty(M) $. Let $ q \in (1,\infty) $ be such that
\begin{equation}\label{e-q}
q > 	\max \left\{ \frac{pm-2}{2-p} , \, \frac{2-pm}{p} \right\} .
\end{equation}
By means of \eqref{sub-poin} and computations analogous to the ones that led to \eqref{eq.deriv}, we obtain:
\begin{equation}\label{eq.deriv-2}
\frac{{\rm d}}{{\rm d}t}\|u(t)\|_{q+1;\nu}^{q+1} \le -\frac{4q(q+1)m}{(m+q)^2\,D^2} \, \|u(t)\|_{\frac{p(q+m)}{2};\nu}^{q+m} \, .
\end{equation}
In view of \eqref{e-q} and of the non-expansivity of the $L^\infty$ norm, a standard interpolation yields
\begin{equation}\label{eq.deriv-3}
\| u(t) \|_{q+1;\nu} \le \| u(t) \|_{\frac{p(q+m)}{2};\nu}^{\frac{p(q+m)}{2(q+1)}} \, \| u_0 \|_{\infty}^{\frac{(2-p)q+2-pm}{2(q+1)}} \quad \forall t>0 \, ,
\end{equation}
so that by combining \eqref{eq.deriv-2} and \eqref{eq.deriv-3} we can deduce
\begin{equation*}\label{eq.deriv-4}
\frac{{\rm d}}{{\rm d}t}\|u(t)\|_{q+1;\nu}^{q+1} \le -\frac{4q(q+1)m}{(m+q)^2\,D^2} \, \frac{\|u(t)\|_{q+1;\nu}^{(q+1) \frac{2}{p}}}{\| u_0 \|_\infty^{\frac{(2-p)q+2-pm}{p}}} \, .
\end{equation*}
Integrating the above differential inequality with respect to the variable $ y(t):=\| u(t) \|_{q+1;\nu}^{q+1} $ we find
\begin{equation*}\label{eq.deriv-5}
y(t)^{\frac{2-p}{p}} \le \frac{1}{\frac{1}{y(0)^{\frac{2-p}{p}}} + \frac{4(2-p)q(q+1)m}{p(m+q)^2 \,D^2 \, \| u_0 \|_\infty^{\frac{(2-p)q+2-pm}{p}}} \, t } \quad \forall t>0 \, ,
\end{equation*}
which implies
\begin{equation}\label{eq.deriv-6}
\| u(t) \|_{q+1;\nu} \le A \, \frac{\| u_0 \|_\infty^{\frac{(2-p)q+2-pm}{(2-p)(q+1)}}}{ t^{\frac{r}{(2-p)(q+1)}} } \quad \forall t>0 \, ,
\end{equation}
where $ A $ is a positive constant depending only on $ q $, $p$, $ D $ and $m$. Using the smoothing effect \eqref{eq.smoothing} (which holds due to \eqref{eq.Sob-wei}, upon replacing $ \| u_0 \|_{q+1} $ with $\| u_0 \|_{q+1;\nu} $) between $ t $ and $t/2$, together with \eqref{eq.deriv-6}, we get
\begin{equation}\label{eq.deriv-7}
\begin{aligned}
\| u(t) \|_{\infty} & \le K \, 2^{\frac{\sigma}{(\sigma-1)(q+1)+\sigma(m-1)}} \, t^{-\frac{\sigma}{(\sigma-1)(q+1)+\sigma(m-1)}} \, \| u(t/2) \|_{q+1;\nu}^{\frac{(\sigma-1)(q+1)}{(\sigma-1)(q+1)+\sigma(m-1)}} \\
& \le A^\prime \, t^{-\frac{2\sigma-p}{(2-p)[(\sigma-1)(q+1)+\sigma(m-1)]}} \, \| u_0 \|_\infty^{\frac{(\sigma-1)[(2-p)q+2-pm]}{(2-p)[(\sigma-1)(q+1)+\sigma(m-1)]}} \quad \forall t>0 \, ,
\end{aligned}
\end{equation}
where $ A^\prime $ is another positive constant depending only on $ q $, $p$, $ D $, $ \sigma $, $K$ and $m$. By shifting the time origin from $0$ to $t/2$ in \eqref{eq.deriv-7} we then infer
\begin{equation}\label{eq.deriv-8}
\begin{aligned}
\| u(t) \|_{\infty} \le &A^\prime \, 2^{\frac{2\sigma-p}{(2-p)[(\sigma-1)(q+1)+\sigma(m-1)]}}\\ & \times t^{-\frac{2\sigma-p}{(2-p)[(\sigma-1)(q+1)+\sigma(m-1)]}} \, \| u(t/2) \|_\infty^{\frac{(\sigma-1)[(2-p)q+2-pm]}{(2-p)[(\sigma-1)(q+1)+\sigma(m-1)]}} \qquad  \forall t>0 \, .
\end{aligned}
\end{equation}
Since
$$ \frac{(\sigma-1)[(2-p)q+2-pm]}{(2-p)[(\sigma-1)(q+1)+\sigma(m-1)]} \in (0,1) \, ,  $$
the validity of \eqref{eq: absb} is just a consequence of a routine iteration of inequality \eqref{eq.deriv-8}.
\end{proof}

\begin{rem}\label{sharp}\rm
The sharpness of the bound \eqref{eq: absb} is standard. In fact, it is enough to pick a positive solution $W_R$ to the sublinear elliptic problem
\begin{equation}\label{elliptic}
\begin{cases}
-\operatorname{div}\left( \rho_\mu \nabla{W_R} \right) = \rho_\nu \, W_R^{\frac{1}{m}} & \textrm{in } B_R \, , \\
W_R=0 & \textrm{on } \partial B_R \, ,
\end{cases}
\end{equation}
for any $ R>0 $, and notice that the function
$$ (t+1)^{-\frac{1}{m-1}} \, W^{1/m}_R(x) \, , $$
set to zero outside $B_R \times \mathbb{R}^{+} $, is a \emph{subsolution} to \eqref{eq-pesata} with initial datum $ u_0=W^{1/m}_R $.
Existence of solutions to problem \eqref{elliptic} is standard given the fact that that $\rho_\nu, \rho_\nu^{-1},\rho_\mu, \rho_\mu^{-1}\in L^\infty_{{\rm loc}}(M)$ and $m>1$.
\end{rem}

\subsection{Examples}\label{exa-4}
We shall discuss here some particular examples of \emph{weighted} manifolds on which the assumptions of Theorem \ref{thm: sub-poin} hold. For simplicity, as in Section \ref{exa-3} (whose notations we take for granted), we shall restrict ourselves to \emph{model manifolds}. So, given two radial weights $ \rho_\nu  $ and $ \rho_\mu $, as a consequence of \cite[Theorem 1.15]{kufner} we have that the sub-Poincar\'e inequality \eqref{sub-poin} holds in $ C^\infty_{c,\rm rad}(M) $ for some $ p \in (1,2) $ if and only if
\begin{equation}\label{cond-sub-poin}\begin{aligned}
\int_0^\infty \left( \int_0^x \rho_\nu(r) \, \psi(r)^{d-1} \, \mathrm{d}r \right)^{\frac{2}{2-p}} \left( \int_x^\infty  \frac{\psi(r)^{1-d}}{\rho_\mu(r)} \, \mathrm{d}r \right)^{\frac{2(p-1)}{2-p}} \frac{\psi(x)^{1-d}}{\rho_\mu(x)} \, \mathrm{d}x < \infty \, .
\end{aligned}\end{equation}
Let us focus on ``hyperbolic'' models, namely we pick $ \psi(r) \sim e^{\frac{c}{d-1} r} $ for large $r$, with $c>0$. As for the weights, one can choose for instance $ \rho_\nu(r)=e^{-\frac{\alpha}{d-1} r} $ and $ \rho_\mu(r)=e^{\frac{\beta}{d-1} r} $, with $ \alpha,\beta \ge 0 $. Straightforward computations show that \eqref{cond-sub-poin} is fulfilled if and only if
\begin{equation*}\label{c-exa-sub}
\max\left\{ 2 \, \frac{c-\alpha}{c+\beta} \, , \ 1 \right\} <p<2 \, .
\end{equation*}
Moreover, the Sobolev inequality \eqref{eq.Sob-wei} holds with $ \sigma=2^\ast/2 $, just as a consequence of the validity of the same inequality on $\mathbb{H}^d$: in fact $\rho_\nu(r)\le1$ and $\rho_\mu(r)\ge1$ for all $r$. We can therefore assert that, with the above choices, the absolute bound \eqref{eq: absb} holds for radial solutions to \eqref{eq-pesata} provided $ 2 \, (c-\alpha)(c+\beta)<2$
which clearly holds iff $ -\alpha < \beta $.

Actually the just established results can be extended to the non-radial setting as well. In fact, we have proved that \eqref{sub-poin} is valid in $ C^\infty_{c,\rm rad}(M) $ with some constant $ D=D_{\rm rad} $. Nevertheless, given any $ f \in C^\infty_c(M) $ (not necessarily radial), we have:
\begin{equation*}\label{cond-sub-poin-nonrad}
\begin{aligned}
\| f \|_{p;\nu} = & \left( \int_{\mathbb{S}_{d-1}} \int_{0}^{+\infty} |f(r,\phi_1,\ldots,\phi_{d-1})|^p \, \psi(r)^{d-1} \, \rho_{\nu}(r) \mathrm{d}r \, \mathrm{d}\Theta  \right)^{\frac{1}{p}} \\
 \le & \left| \mathbb{S}_{d-1} \right|^{\frac{1}{2}-\frac{1}{p}} \, D_{\rm rad} \left[ \int_{\mathbb{S}_{d-1}} \left( \int_{0}^{+\infty} \left|\frac{\partial f}{\partial r}(r,\phi_1,\ldots,\phi_{d-1})\right|^2 \, \psi(r)^{d-1} \, \rho_{\mu}(r) \mathrm{d}r \right)^{\frac{p}{2}} \mathrm{d}\Theta  \right]^{\frac{1}{p}} \\
 \le &   D_{\rm rad} \, \| \nabla{f} \|_{2;\mu}
\end{aligned}
\end{equation*}
by H\"older inequality, which can be used because of the crucial assumption $p<2$. As a consequence we can deduce that, with the above choices, the sub-Poincar\'e inequality \eqref{sub-poin} holds in the whole $ C^\infty_{c}(M)$, so that every solution to \eqref{eq-pesata} (not necessarily radial) satisfies the absolute bound \eqref{eq: absb}.

We summarize the above calculations in the following result.

\begin{prop}
Let $M$ be a model manifold associated with a function $\psi$ satisfying $ \psi(r) \sim e^{\frac{c}{d-1} r} $ as $r\to+\infty$, for some $c>0$. Assume that the weights appearing in the weighted porous medium equation \eqref{eq-pesata} are given by $ \rho_\nu(r)=e^{-\frac{\alpha}{d-1} r} $ and $ \rho_\mu(r)=e^{\frac{\beta}{d-1} r} $, with $ \alpha,\beta \ge 0 $ and $-\alpha<\beta$. Then the sub-Poincar\'e inequality \eqref{sub-poin} and the Sobolev inequality \eqref{eq.Sob-wei} hold.

As a consequence, the absolute bound
\[
\left\| u(t) \right\|_\infty \le B \, t^{-\frac{1}{m-1}} \quad \forall t>0
\]
holds for any solution $u(t)$ to \eqref{eq-pesata} with the present choice of $M$ and of $\rho_\mu, \rho_\nu$.
\end{prop}
%Actually the just established results can be extended to the non-radial setting as well. In fact, thanks to \cite[Example 21.12]{kufner}, we can infer that under condition \eqref{c-exa-sub} inequality \eqref{sub-poin} holds with $ \rho_\nu(r)=r^{d-1} e^{-\frac{\alpha}{d-1} r} $ and $ \rho_\mu(r)=r^{d-1} e^{\frac{\beta}{d-1} r} $ (for large $r$). Given $ \alpha,\beta \ge 0 $ and $p$ complying with \eqref{c-exa-sub}, it is then apparent that there exists $ \varepsilon>0 $ so small that \eqref{c-exa-sub} is still satisfied with $ \beta $ replaced by $ \beta-\varepsilon $, so that in view of \cite[Example 21.12]{kufner} inequality \eqref{sub-poin} holds with $ \rho_\nu(r)=r^{d-1} e^{-\frac{\alpha}{d-1} r} $ and $ \rho_\mu(r)=r^{d-1} e^{\frac{\beta-\varepsilon}{d-1} r} $ (for large $r$); this trivially implies that the same is true with $ \rho_\nu(r)=e^{-\frac{\alpha}{d-1} r} $ and $ \rho_\mu(r)=e^{\frac{\beta}{d-1} r} $ (for large $r$).

\begin{rem}\rm
We point out that, in the above examples, the (weighted) volume of the manifold is infinite provided $ \alpha\le c $. Hence, in such cases, no previous result could be exploited in order to deduce the absolute bound.
\end{rem}

\subsection{An existence result for a weighted sublinear elliptic equation}

\begin{thm}
Let $ \rho_\nu $ and $\rho_\mu$ two weights such that $\rho_\nu, \rho_\nu^{-1},\rho_\mu, \rho_\mu^{-1}\in L_{{\rm loc}}^\infty(M)$ and, moreover, the sub-Poincar\'e inequality \eqref{sub-poin} and the Sobolev-type inequality \eqref{eq.Sob-wei} hold for some $ p \in [1,2) $ and $ \sigma > 1 $. Then there exists a bounded, nontrivial, nonnegative solution $ W $ to the following weighted sublinear elliptic equation:
\begin{equation}\label{sublinear}
-\operatorname{div}\left( \rho_\mu \nabla{W} \right) = \rho_\nu \, W^{\frac{1}{m}} \quad \textrm{in } M \, .
\end{equation}
In addition, $W$ is minimal in the class of bounded, nontrivial, nonnegative very weak solutions of \eqref{sublinear}.
\end{thm}
\begin{proof}
Let $ u $ be the solution to \eqref{eq-pesata} corresponding to some nontrivial initial datum $ u_0 \ge 0 $, with $ u_0 \in L^1(M) $. Let $ U(x,t)=t^{\frac{1}{m-1}}u(x,t) $. The absolute bound \eqref{eq: absb} shows that $ \| U(t) \|_\infty \le B $ for a suitable $B>0$ and all $t>0$. By the classical B\'enilan-Crandall inequality (see e.g.\ \cite{V07}) we know that $ U $ is nondecreasing as a function of $t$; in particular, since $u$ is nonnegative and nontrivial, the limit $ W^{\frac1m}(x)=\lim_{t\to\infty} U(x,t) $ exists and is also nonnegative, nontrivial and bounded. By proceeding as in \cite{Va04}, it is easy to show that $W$ solves \eqref{sublinear}.

As concerns minimality, consider a given very weak solution $\widehat W$, having the stated properties, to \eqref{sublinear}. Consider also the solution $u_R$ to the homogeneous Dirichlet problem, on the ball $B_R$, for the weighted porous medium equation corresponding to the datum $u_0\vert_{B_R}$. An analogue on manifolds of \cite[Theorem 6.5]{V07} yields comparison between $u_R$ and the supersolution $t^{-\frac1{m-1}}\widehat W^{\frac1m}\vert_{B_R}$. Hence, letting $R\to+\infty$ and noting that by construction $u_R\to u$ e.g.\ pointwise, we get
\[
t^{\frac1{m-1}}u(x,t)\le \widehat W^{\frac1m}(x),\ \ {\rm for\ a.e.}\ x\in M,\ t>0,
\]
whence the claim follows by letting $t\to+\infty$ and using the first part of the proof.
\end{proof}

For related results in the non-weighted case see the classical paper \cite{BK}. The case of sublinear elliptic equations posed on the hyperbolic space is briefly discussed in \cite{BGGV}.

\subsection{Sub-Poincar\'e inequalities fail on Cartan-Hadamard manifolds}\label{sec-no-sub}

One may wonder whether examples similar to the ones given above can be shown on manifolds without imposing that suitable weights are present. Surprisingly enough, this is \it never \rm the case at least when Cartan-Hadamard manifolds are concerned. In fact, we have the following result, of independent interest.

\begin{thm}\label{c-h-sub}
Let $ M $ be a Cartan-Hadamard manifold and let $ p \in [1,2) $. Then there exists no positive constant $ C>0 $ for which the sub-Poincar\'e inequality
\begin{equation}\label{sub-1}
\| f \|_p \le C \left\| \nabla{f} \right\|_{2} \quad \forall f \in C_c^\infty(M)
\end{equation}
holds.
\end{thm}
\begin{proof}
We shall proceed by contradiction. Indeed, take any $ r_0 \in (0,\infty) $. Let $ \delta: (0,\infty) \mapsto (0,\infty) $ be a suitable function that will be specified later, and $ \xi : [0,\infty) \mapsto [0,1] $ be a regular non-increasing function such that
\begin{equation}\label{eq-delta-1}
\xi(r) = 1 \ \ \ \forall r \in [0,r_0) \, , \quad \xi(r)=0 \ \ \ \forall r \in [r_0+\delta(r_0),\infty) \, , \quad \xi(r_0+\delta(r_0)/2)=1/2 \, ,
\end{equation}
\begin{equation}\label{eq-delta-2}
\left| \xi^\prime(r) \right| \le \frac{2}{\delta(r_0)} \ \ \ \forall r \in [0,\infty) \, .
\end{equation}
Now choose any point $ o \in M $ and consider it as a pole. As usual we denote by $ B_R $ the geodesic ball of radius $R$ centered at $o$, and by $ V(R) $ its volume. As a matter of fact,
\begin{equation}\label{eq-delta-3}
x \mapsto \operatorname{dist}(x,o) \in C^\infty(M \setminus \{o\}) \, , \quad \left| \nabla_x \,{ \operatorname{dist}(x,o) } \right| = 1 \quad \forall x \in M \setminus \{o\} \, .
\end{equation}
If \eqref{sub-1} holds, we are then allowed to pick the test function
\begin{equation*}\label{eq-delta-4}
f(x)=\xi\left( \operatorname{dist}(x,o) \right) \quad \forall x \in M \, ,
\end{equation*}
which yields
\begin{equation}\label{eq-delta-5}
\left( \int_{B_{r_0+\delta(r_0)}} \left|f(x)\right|^p \mathrm{d}x  \right)^{\frac{1}{p}} \le C \left( \int_{B_{r_0+\delta(r_0)} \setminus B_{r_0} } \left| \nabla f(x) \right|^2 \mathrm{d}x  \right)^{\frac{1}{2}} \, .
\end{equation}
In view of \eqref{eq-delta-1}, \eqref{eq-delta-2}, \eqref{eq-delta-3} and the fact that $ \xi $ is non-increasing, from \eqref{eq-delta-5} we deduce that
\begin{equation}\label{eq-delta-6}
\frac{1}{2} \left[ V(r_0+\delta(r_0)/2) - V(r_0) \right]^{\frac{1}{p}} \le \frac{2C}{\delta(r_0)} \left[ V(r_0+\delta(r_0)) - V(r_0) \right]^{\frac{1}{2}} \, .
\end{equation}
Since
$$  \frac{\mathrm{d}V(R)}{\mathrm{d}R} = S(R) \quad \forall R>0 \, , $$
where $S(R)$ is the Riemannian measure of $ \partial B_R $ (see e.g.\ \cite[Section 3]{G4}), we can rewrite \eqref{eq-delta-6} as
\begin{equation}\label{eq-delta-7}
\left( \int_{r_0}^{r_0+\delta(r_0)/2} S(r) \, \mathrm{d}r \right)^{\frac{1}{p}} \le \frac{4C}{\delta(r_0)} \left( \int_{r_0}^{r_0+\delta(r_0)} S(r) \, \mathrm{d}r \right)^{\frac{1}{2}} \, .
\end{equation}
Because $M$ is a Cartan-Hadamard manifold, it is well known that the function $ R \mapsto S(R) $ is regular, strictly increasing and such that $ \lim_{R\to\infty} S(R)=\infty $ (see once again \cite[Section 3, Section 15]{G4}. It is therefore possible to choose $ \delta $ as
\begin{equation*}\label{eq-delta-8}
\delta(r)=S^{-1}\left[ 2 S(r) \right] -r \quad \forall r>0 \, ,
\end{equation*}
so that
\begin{equation}\label{eq-delta-9}
S(r_0) \le S(r) \le 2 \, S(r_0) \quad \forall r \in [r_0,r_0+\delta(r_0)] \, .
\end{equation}
In particular, \eqref{eq-delta-7} and \eqref{eq-delta-9} imply
\begin{equation*}\label{eq-delta-10}
\left[ S(r_0) \, \frac{\delta(r_0)}{2} \right]^{\frac{1}{p}} \le \frac{4C}{\delta(r_0)} \left[ 2 \, S(r_0) \delta(r_0) \right]^{\frac{1}{2}} \, ,
\end{equation*}
that is
\begin{equation}\label{eq-delta-11}
 S(r_0)^{\frac{2-p}{2+p}} \, \left[ S^{-1}\left[ 2 S(r_0) \right] -r_0 \right] \le 2 (4C)^{\frac{2p}{2+p}} =: C^\prime \, .
\end{equation}
As \eqref{eq-delta-11} holds for all $ r_0 \in (0,\infty) $, we can let $ r_0=S^{-1}(s_0) $ for any $s_0 \in (0,\infty)$, to get
\begin{equation}\label{eq-delta-12}
S^{-1}\left[ 2 s_0 \right] \le S^{-1}(s_0) + \frac{C^\prime}{s_0^{\frac{2-p}{2+p}}} \quad \forall s_0 \in (0,\infty) \, .
\end{equation}
A straightforward iteration of \eqref{eq-delta-12} yields
\begin{equation}\label{eq-delta-13}
S^{-1}\left( s_0 \right) \le S^{-1}\left(s_0/2^n\right) +  \frac{C^\prime}{s_0^{\frac{2-p}{2+p}}} \, \frac{2^{n\frac{2-p}{2+p}}-1}{1-2^{-\frac{2-p}{2+p}}} \quad \forall s_0 \in (0,\infty) \, , \ \forall n \in \mathbb{N} \, .
\end{equation}
By setting $ n= \log_2(s_0)\rceil $ in \eqref{eq-delta-13}, we find
\begin{equation}\label{eq-delta-14}
S^{-1}\left( s_0 \right) \le S^{-1}\left(1\right) +  \frac{C^\prime}{s_0^{\frac{2-p}{2+p}}} \, \frac{2^{\frac{2-p}{2+p}} \, s_0^{\frac{2-p}{2+p}} -1}{1-2^{-\frac{2-p}{2+p}}} \quad \forall s_0 \in (0,\infty) \, , \ \forall n \in \mathbb{N} \, .
\end{equation}
Letting $ s_0 \to \infty $ in \eqref{eq-delta-14} end up with
$$  \limsup_{s_0 \to \infty} S^{-1}\left( s_0 \right) < \infty \, , $$
a contradiction since $ R \to S(R) $ is monotone increasing.
\end{proof}

\begin{rem}\rm
We point out that in the above result the condition $ p < 2 $ is essential. In fact, on the one hand \eqref{sub-1} holds on any Cartan-Hadamard manifold when $ p=2^\ast $ (let $ d \ge 3 $), on the other hand the proof of Theorem \ref{c-h-sub} clearly fails when $ p \ge 2 $, since the iteration procedure following \eqref{eq-delta-12} does not yield any contradiction in that case.
\end{rem}

% In questa sezione dire che ci mettiamo nel contesto pesato perche' nella letteratura non ci sono risultati che parlino di disuguaglianze di sub-Poincaré nel caso di varietà di Cartan-Hadamard. Dire anche che i risultati delle sezioni 2 e 3 si applicano ovviamente anche al caso pesato, perché basta assumere la validità di una sobolev+poincaré o di una famiglia di sobolev che degenerano.

%%%%%%%%%%%%%%%%%%%%%%%%%%

%\newpage

\par\bigskip\noindent
\textbf{Acknowledgments.} The authors were partially supported by the PRIN project {\em Equazioni alle derivate parziali di tipo ellittico e parabolico: aspetti geometrici, disuguaglianze collegate, e applicazioni}. They were also supported by the Gruppo Nazionale per l'Analisi Matematica, la Probabilit\`a e le loro Applicazioni (GNAMPA) of the Istituto Nazionale di Alta Matematica (INdAM).

\end{document}